\documentclass[11pt,epsf]{article}
\usepackage{graphicx}
\usepackage{amsthm}
\usepackage{amsfonts}
\usepackage{amssymb}
\title{On the complex conjugate zeros of the partial theta function}
\author{Vladimir Petrov Kostov\\ 
Universit\'e C\^ote d'Azur, CNRS, LJAD, France\\   
e-mail: vladimir.kostov@unice.fr} 
\date{}
\newtheorem{tm}{Theorem}

\newtheorem{rem}[tm]{Remark}
\newtheorem{rems}[tm]{Remarks}
\newtheorem{lm}[tm]{Lemma}

\begin{document} 
\maketitle 
\begin{abstract}
We prove that 1) for any $q\in (0,1)$, all complex conjugate pairs of zeros 
of the partial theta function $\theta (q,x):=\sum _{j=0}^{\infty}q^{j(j+1)/2}x^j$ 
belong to the set 
$\{$~Re\,$x\in (-5792.7,0),$~$|$Im\,$x|<132~\}$ $\cup$ 
$\{ ~|x|<18~\}$ and 2) for any $q\in (-1,0)$, they belong to the rectangle 
$\{$~$|$Re\,$x|< 364.2,$~$|$Im\,$x|<132~\}$.\\ 

{\bf Key words:} partial theta function, Jacobi theta function, 
Jacobi triple product\\ 

{\bf AMS classification:} 26A06 
\end{abstract}

\section{Introduction}

For any fixed $q\in \mathbb{C}$ ($|q|<1$), the sum of the bivariate series 
$\theta (q,x):=\sum _{j=0}^{\infty}q^{j(j+1)/2}x^j$ is an entire function in $x$ 
called the {\em partial theta function}. We remind that the Jacobi theta 
function equals $\Theta (q,x):=\sum _{j=-\infty}^{\infty}q^{j^2}x^j$ and that 
$\theta (q^2,x/q)=\sum _{j=0}^{\infty}q^{j^2}x^j$. 
We consider the case when $q$ is real. 
In the present note we prove the following theorem:

\begin{tm}\label{maintm}
(1) For any $q\in (0,1)$, all complex conjugate pairs of zeros 
of $\theta (q,.)$  
belong to the set 
$\{$~{\rm Re}\,$x\in (-5792.7,0),$~$|${\rm Im}\,$x|<132~\}$ $\cup$ 
$\{ ~|x|<18~\}$.

(2) For any $q\in (-1,0)$, all complex conjugate pairs of zeros 
of $\theta (q,.)$  
belong to the set $\{$~$|${\rm Re}\,$x|< 364.2,$~$|${\rm Im}\,$x|<132~\}$.
\end{tm}

One of the most recent applications of $\theta$ is connected with a problem on 
{\em section-hyperbolic 
polynomials}, i.e. real univariate polynomials with all 
roots real and such that all their truncations of positive degree have also 
all their roots real. Inspired by classical results of Hardy, Petrovitch and 
Hutchinson (see \cite{Ha}, \cite{Pe} and \cite{Hu}), the study of the problem 
was continued in \cite{Ost}, \cite{KaLoVi} and \cite{KoSh}. 
Other areas of application of $\theta$ are modularity and asymptotics of 
regularized characters and 
quantum dimensions of the $(1,p)$-singlet algebra modules (see \cite{BFM} 
and \cite{CMW}), the theory 
of (mock) modular forms (see \cite{BrFoRh}), Ramanujan-type $q$-series 
(see \cite{Wa}), asymptotic analysis (see \cite{BeKi}) and statistical physics 
and combinatorics (see \cite{So}); see also~\cite{AnBe}.

The following properties of $\theta$ are proved in \cite{KoBSM1}: {\em 
For each $q\in (0,1)$, the function $\theta (q,.)$ has infinitely-many negative 
and no nonnegative zeros. There exists a sequence of values of $q$ (tending to 
$1^-$) $0=:\tilde{q}_0<\tilde{q}_1=0.3092\ldots <\tilde{q}_2=0.5169\ldots 
<\cdots <1$ 
for which 
$\theta$ has exactly one multiple zero. This is the rightmost of its (negative) 
zeros; it is a double one. For $q\in (\tilde{q}_j,\tilde{q}_{j+1}]$, 
$\theta (q,.)$ has exactly $j$ complex conjugate pairs of zeros counted with 
multiplicity.}

Analogous properties for $q\in (-1,0)$ read (see~\cite{KoPRSE2}): 
{\em For each $q\in (-1,0)$, the function $\theta (q,.)$ has infinitely-many 
negative 
and infinitely-many positive zeros. 
There exists a sequence of values $\bar{q}_j$ of $q$ (tending to 
$-1^+$, the largest of which is $\bar{q}_1=-0.727133\ldots$) 
for which 
$\theta$ has exactly one multiple zero; it is a double one. This is the 
rightmost of its negative 
zeros for $j$ odd and the second from the left of its positive zeros 
for $j$ even. For $q\in [\bar{q}_1,0)$, all zeros of $\theta$ are real.}

\section{Proof of Theorem~\protect\ref{maintm}}

\begin{proof}[Proof of part (1)]
We write ``CCP'' for ``complex conjugate pair (of zeros of $\theta (q,.)$)''. 
We use a result of \cite{KoPMD}: {\em For any $q\in (0,1)$, all zeros of 
$\theta (q,.)$ belong to the domain 
$\{ {\rm Re}~x<0, |{\rm Im}~x|<132\}$$\cup$$\{ {\rm Re}~x\geq 0, 
|x|<18\}$.} To prove part (1) of Theorem~\ref{maintm} means to show that the 
real parts of the CCPs are $>-5792.7$. 
The proof is based on a comparison between the functions $\theta$,  
$\Theta ^*:=\sum _{j=-\infty}^{\infty}q^{j(j+1)/2}x^j$ and 
$G:=\sum _{j=-\infty}^{-1}q^{j(j+1)/2}x^j$. Obviously, $\theta =\Theta ^*-G$. For the 
function $\Theta ^*$ we use the following formula (derived from the 
Jacobi triple product, see~\cite{KoFAA}):

$$\Theta ^*(q,x)=QPR~~~\,{\rm ,~where}~~~\, 
Q:=\prod _{m=1}^{\infty}(1-q^m)~~~,~~~P:=\prod _{m=1}^{\infty}(1+xq^m)~~~
{\rm and}~~~R:=\prod _{m=1}^{\infty}(1+q^{m-1}/x)~.$$
Consider the functions $\theta$, $\Theta ^*$ and $G$ restricted  
to the vertical line in the 
$x$-plane  
$\mathcal{L}_{\nu}(q)~:~$Re\,$x=-q^{-\nu -1/2}$, $\nu \in \mathbb{N}$ which  
avoids the zeros of $\Theta ^*$. The  
product $|Q|\cdot |P|\cdot |R|$ 
is minimal for $x=-q^{-\nu -1/2}$ 
(because for $g:=q^{m-1}$, $u<-1$ and $v\in \mathbb{R}^*$, one has 
$|1+g/(u+iv)|^2-|1+g/u|^2=-(2u+g)v^2g/(u^2(u^2+v^2))>0$ and 
$|1+(u+iv)q^m|>|1+uq^m|$). Obviously, 

\begin{equation}\label{majorizeG}
|G|\leq \sum _{j=0}^{\infty}|q|^{j(j-1)/2}|x|^{-j-1}\leq 
\sum _{j=0}^{\infty}|x|^{-j-1}=1/(|x|-1)~.
\end{equation}
When $x\in \mathcal{L}_{\nu}(q)$, the right-hand side is maximal for 
$x=-q^{-\nu -1/2}$. 
If for given $q$ and  
$\nu$ one has  

\begin{equation}\label{ineqThetaG}
|\Theta ^*(q,-q^{-\nu -1/2})|>|G(q,-q^{-\nu -1/2})|~,
\end{equation}
then the inequality $|\Theta ^*|>|G|$, 
hence  
$\Theta ^*-G\neq 0$, holds true along $\mathcal{L}_{\nu}(q)$ and 
the function $\theta$ has no zeros on the line $\mathcal{L}_{\nu}(q)$. 
We set $[\tilde{q}_1,1)=\cup _{n=1}^{\infty}K_n$, where 
$K_n:=[\tilde{q}_1,1-1/(n+1)]$.
We set also $\ell _1:=4$ and $\ell _n:=4(n+1)$, 
if $n\geq 2$. 
\begin{lm}\label{lmLnu}
For $q\in K_n$, $\nu \geq \ell _n$ and $x\in \mathcal{L}_{\nu}(q)$, 
one has $|\Theta ^*|>|G|$ 
(hence $\theta \neq 0$).
\end{lm}

\begin{proof}For $n=1$, one has $Q\geq \prod _{j=1}^{\infty}(1-2^{-j})=
0.288\ldots$ and for $\nu \in \mathbb{N}$, $\nu \geq \ell _1$, 
$R|_{x=-q^{-\nu -1/2}}>Q$. Denote by 
$P^{\bullet}$ the product $P|_{x=-q^{-\nu -1/2}}$ and set 
$P^{\dagger}:=\prod _{j=1}^{\infty}(1-q^{j-1/2})$. Thus

$$P^{\bullet}\geq P^{\flat}P^{\dagger}~~~,\, \, {\rm where}~~~
P^{\flat}:=(1-q^{-7/2})(1-q^{-5/2})(1-q^{-3/2})(1-q^{-1/2})$$
with equality only for $\nu =\ell _1=4$. 
Both factors $P^{\flat}$ and $P^{\dagger}$ can be minorized by their 
respective values $36.3\ldots$ and $0.129\ldots$
for $q=1/2$, so 
$P^{\bullet}\geq P^{\bullet}|_{\nu =4,q=1/2}>36.3\cdot 0.129=4.68\ldots$ and 

$$|\Theta ^*(q,-q^{-\nu -1/2})|>4.68\cdot 0.288^2=0.388\ldots~.$$
At the same time, for $x=-q^{-\nu -1/2}\leq -q^{-9/2}$ and for 
$q\in [\tilde{q}_1,1/2]$, 
the majoration $1/(|x|-1)$ of $|G|$ (see (\ref{majorizeG}))
is maximal for $q=1/2$, $\nu =4$, 
in which case it equals $0.046\ldots <0.388\ldots$. 

For $n\geq 2$, 
we use the inequality $Q\geq e^{-(\pi ^2/6)n}$ 
which holds true for $q\in K_n$, 
see Lemma~4 in~\cite{KoFAA}. As in the case $q\in K_1$, one obtains 
$R|_{x=-q^{-\nu -1/2}}>Q$, $\mathbb{N}\ni \nu \geq \ell _n$. 
The product $P$ is represented in the form 
$P=P^{\triangle}P^{\dagger}$ (with $P^{\dagger}$ as above), where 

$$P^{\triangle}:=\prod _{m=1}^{\nu}(1-q^{-\nu -1/2+m})=
q^{-\nu ^2/2}(-1)^{\nu}P^{\sharp}~~~,~~~
P^{\sharp}:=\prod _{m=1}^{\nu }(1-q^{\nu +1/2-m})~.$$
It is clear that $P^{\sharp}>P^{\dagger}>(1-\sqrt{q})Q$; the rightmost 
inequality follows from $1-q^{j-1/2}<1-q^j$, $j\in \mathbb{N}$. 
Set $h:=1-1/(n+1)$, $f^*:=(1-1/x)^{-x}$, and $\tau :=(1-h^{1/2})^2$.
The factor 
$q^{-\nu ^2/2}$ is minorized by 
$h^{-8(n+1)^2}=((f^*)(n+1))^{8(n+1)}$.
The function $f^*$ is decreasing on $[2,\infty )$, 
with $f^*(2)=4$ and $\lim _{x\rightarrow \infty}f^*(x)=e$, 
therefore $q^{-\nu ^2/2}\geq ((f^*)(n+1))^{8(n+1)}>e^{8(n+1)}$. 
Hence 

$$\begin{array}{c}
\tau =((1/(n+1))/(1+h^{1/2}))^2>1/4(n+1)^2~~~~\,  
{\rm and}\\ \\ 
|\Theta ^*|=|QPR|>Q^2|P|
=Q^2|P^{\triangle}P^{\dagger}|>Q^2(P^{\dagger})^2q^{-\nu ^2/2}>
Q^4(1-\sqrt{q})^2q^{-\nu ^2/2}\\ \\ >Q^4\tau e^{8(n+1)}>
e^{-4(\pi ^2/6)n}\tau e^{8(n+1)}=
e^{(8-2\pi ^2/3)n+8}\tau >e^{n+8}/4(n+1)^2~.\end{array}$$
The function $e^{x+8}/4(x+1)^2$ is increasing on $[2,\infty)$, so  
$|\Theta ^*|>e^{10}/36$.  
For $|G|$ one obtains

$$\begin{array}{ccccccc}
|G|&\leq &1/(q^{-\nu -1/2}-1)&\leq &1/(h^{-\nu -1/2}-1)&\leq&
1/(h^{-4(n+1)-1/2}-1)\\ \\&\leq &1/(e^4h^{-1/2}-1) 
&\leq &1/(e^4-1)&<&e^{10}/36~.\end{array}$$
\end{proof}

\begin{lm}\label{lmI}
For $j\in \mathbb{N}$, the double zero of $\theta (\tilde{q}_j,.)$ 
belongs to the interval $I:=[-1226,0]$.
\end{lm}

\begin{proof}We use a result of~\cite{KoAA}: 
{\em Consider the function $\theta$ 
with $q$, $x\in \mathbb{C}$, $|q|<1$. Set 
$\alpha _0:=\sqrt{3}/(2\pi )=0.27\ldots$. Then for $k\geq n\geq 5$ 
and for $|q|\leq 1-1/(\alpha _0n)$, there exists exactly one (and simple) 
zero $\xi _k$ of 
$\theta (q,.)$ satisfying the conditions 
$|q|^{-k+1/2}<|\xi _k|<|q|^{-k-1/2}$, no zero is of modulus  
$|q|^{-k\pm 1/2}$, and exactly $n-1$ zeros 
counted with multiplicity are of modulus $<|q|^{-n+1/2}$.}

Hence for $|q|\in [1-1/(\alpha _0(n-1)),1-1/(\alpha _0n)]$, for any multiple zero 
$\xi$ of $\theta$ one has $|\xi |\leq \gamma _n:=(1-1/(\alpha _0(n-1)))^{-n+1/2}$. 
Indeed, 
the quantity $|q|^{-n+1/2}$ is maximal for $|q|=1-1/(\alpha _0(n-1))$ and 
no multiple zero of $\theta$ is of modulus $\geq |q|^{-n+1/2}$. 
For $n\geq 6$, the 
sequence $\gamma _n$ is decreasing, $\gamma _6=1225.1\ldots$ and 
$1-1/(5\alpha _0)=0.27\ldots <\tilde{q}_1$, so for $q\in (0,1)$, 
any double zero of $\theta$ is in~$I$.
\end{proof}

Simple zeros of $\theta$ (real or complex) depend continuously on $q$. 
Two simple real zeros of $\theta (\tilde{q}_j^-,.)$ coalesce for $q=\tilde{q}_j$ to become 
a complex pair for $q=\tilde{q}_j^+$ (as we say, a pair born from the double zero 
of $\theta (\tilde{q}_j,.)$), see \cite{KoBSM1}. 
For each $q$ fixed, the line $\mathcal{L}_{\nu}(q)$ is to the right of 
$\mathcal{L}_{\nu +1}(q)$. 
As $q$ increases, the number $-q^{-\nu -1/2}$ increases and $\mathcal{L}_{\nu}(q)$ moves 
from left to right.  
For any $q\in K_1$, the line $\mathcal{L}_{\nu}(q)$, $\nu \geq \ell _1=4$, 
is to the left of the double zero $-7.5\ldots$  
of $\theta (\tilde{q}_1,.)$, see \cite{KoSh} or~\cite{KoBSM1}; 
one has $\mathcal{L}_4(\tilde{q}_1)~:~$Re\,$x=-196.7\ldots$ and 
$\mathcal{L}_4(1/2)~:~$Re\,$x=-22.6\ldots$. Recall that 
$\tilde{q}_2=0.51\ldots \not\in K_1$. Hence for 
$q\in K_1\setminus \tilde{q}_1$, there is exactly one CCP, 
the one born from the double zero of $\theta (\tilde{q}_1,.)$.  
For $q\in K_1$, there are no zeros of $\theta$ on $\mathcal{L}_4(q)$, so as $q$ 
increases from $\tilde{q}_1$ to $1/2$, the line 
$\mathcal{L}_4(q)$ does not encounter the only CCP of $\theta$ and thus 
there are no CCPs on and to the left of~$\mathcal{L}_4(q)$. 

Thus 
there are no CCPs on and to the left 
of $\mathcal{L}_{\ell _2}(1/2)=\mathcal{L}_{12}(1/2)$. 
One 
has $\mathcal{L}_{12}(1/2)~:~$Re\,$x=-5792.6\ldots$,  
$\mathcal{L}_{12}(0.56)~:~$Re\,$x=-1404.9\ldots$ and  
$\mathcal{L}_{14}(0.56)~:~$Re\,$x=-4479.9\ldots$ (here and below the 
values 
$0.56$, $0.6$ and $0.63$ of $q$ are chosen for the convenience of computation). 
As $q$ increases from $1/2$ to $0.56$, the line $\mathcal{L}_{12}(q)$ 
moves from Re\,$x=-5792.6\ldots$ 
to Re\,$x=-1404.9\ldots$ and encounters no zeros of $\theta$ on its way. 
As CCPs can be 
born only on the interval $I$, for $q\in [1/2,0.56]$, 
no CCPs are born to the left of $\mathcal{L}_{12}(q)$. 
For $q=0.56$, there are no CCPs 
with Re\,$x\leq -4479.9\ldots$, and for $q\in [\tilde{q}_1,0.56]$, 
there are no CCPs 
with Re\,$x\leq -5792.6\ldots$. 

One has $\mathcal{L}_{14}(0.6)~:~$Re\,$x=-1647.4\ldots$ and 
$\mathcal{L}_{16}(0.6)~:~$Re\,$x=-4576.1\ldots$. For $q\in [\tilde{q}_1,0.6]$, 
the line $\mathcal{L}_{14}(q)$ encounters no zeros 
of $\theta$ and it does not intersect $I$, 
so there are no 
CCPs on and to the left of it. 
One has $\mathcal{L}_{16}(0.63)~:~$Re\,$x=-2045.8\ldots$ and 
$\mathcal{L}_{18}(0.63)~:~$Re\,$x=-5154.6\ldots$. Thus there are no CCPs 
on and to the left of $\mathcal{L}_{16}(q)$ for $q\in [0.6,0.63]$. 
One has $\mathcal{L}_{18}(2/3)~:~$Re\,$x=-1810.0\ldots$, 
therefore there are no CCPs on and to the left of $\mathcal{L}_{18}(q)$ for 
$q\in [0.63,2/3]$, hence 
there are no CCPs in Re\,$x\leq -5792.6\ldots$ for $q\in K_2$.

Consider the numbers $b_n:=-(1-1/n)^{-4(n+1)-1/2}$, $\mathbb{N}\ni n\geq 2$. 
These are the values of 
$-q^{-\ell _n-1/2}$ for $q=1-1/n\in K_n$. The sequence $b_n$ 
is increasing, with $b_2=-5792.6\ldots$ and $b_3=-804.4\ldots$. Hence for 
$n\geq 3$, $b_n\in I$. For $q\in [2/3,1)$, there is at least one number 
$-q^{-\nu _0-1/2}$, $\nu _0\in \mathbb{N}$, 
in the interval $(-5792.6\ldots ,-1226)$; this follows from 
$1226\cdot (3/2)<5792.6$. The corresponding line $\mathcal{L}_{\nu _0}(q)$ 
contains no zeros of $\theta$ for $q\in K_n$; this results from 
Lemma~\ref{lmLnu} and from $b_n\in I$ for $n\geq 3$. For each line 
$\mathcal{L}_{\nu}(q)$, $\nu \geq \nu _0$, it is true 
that for $q\in K_n$, there are no CCPs on and to the left of it. 
As $q$ increases in $K_n$, the line $\mathcal{L}_{\nu _0}(q)$ 
might begin to intersect the interval $I$, 
but it is always true that any line $\mathcal{L}_{\nu _0+k}(q)$, 
$k\in \mathbb{N}$, which does not intersect $I$ 
has no CCPs on it and to its left (and there is at least one such line 
for Re\,$x\geq -5792.6\ldots$). Hence there are no CCPs for 
$q\in [\tilde{q}_1,1)$ and Re\,$x\leq -5792.6\ldots$.   
\end{proof}

\begin{proof}[Proof of part (2)] 
We use again a result of \cite{KoPMD}: {\em For any $q\in (-1,0)$, all zeros 
of $\theta (q,.)$ belong to the strip $\{ |{\rm Im}~x|<132\}$.} We have to show 
that the real part of any 
complex zero of $\theta$ belongs to the interval $(-364.2,364.2)$. 
For $q\in (-1,0)$, we consider in the same way the intervals 
$J_n:=[-1+1/(n+1),\bar{q}_1]$, $n\geq 3$. We set $\ell _n:=4(n+1)$. 
For $q\in J_n$, we consider the vertical lines 
$\mathcal{L}_{\nu}(q)~:~$Re\,$x=-|q|^{-\nu -1/2}$ and 
$\mathcal{R}_{\nu}(q)~:~$Re\,$x=|q|^{-\nu -1/2}$, $\mathbb{N}\ni \nu\geq \ell _n$. 
As $|q|$ increases, the line $\mathcal{L}_{\nu}(q)$ moves to the right and $\mathcal{R}_{\nu}(q)$ 
moves to the left. 
We compare the factors $\mu :=1-q^m$, $\lambda :=1+xq^m$ and $\chi :=1+q^{m-1}/x$ 
in cases A) $q=q_*\in (0,1)$, Re\,$x=-(q_*)^{-\nu -1/2}$ and 
B) $q=-q_*\in (-1,0)$, Re\,$x=\pm |q_*|^{-\nu -1/2}$; we denote them by 
$\mu _A$, $\mu _B$, $\lambda _A$, $\lambda _B$, $\chi _A$ or $\chi _B$ 
according to the case. Clearly, $|$Im\,$\lambda _A|=|$Im\,$\lambda _B|$, 
$|$Im\,$\chi _A|=|$Im\,$\chi _B|$ and either $\mu _B=1-q_*^m=\mu _A\in (0,1)$ or 
$\mu _B=1+q_*^m>1>\mu _A$, either $|$Re\,$\lambda _B|=|1-q_*^{-\nu -1/2+m}|=|$Re\,$\lambda _A|$ or 
$|$Re\,$\lambda _B|=|1+q_*^{-\nu -1/2+m}|>|$Re\,$\lambda _A|$ and either 
$|$Re\,$\chi _B|=|1-q_*^{-\nu -3/2+m}/|x|^2|=|$Re\,$\chi _A|$ or 
$|$Re\,$\chi _B|=|1+q_*^{-\nu -3/2+m}/|x|^2|>|$Re\,$\chi _A|$.
%
%
Hence $|\Theta ^*(-q_*,\pm |q_*|^{-\nu -1/2})|>
|\Theta ^*(q_*,-(q_*)^{-\nu -1/2})|$.  

The majoration (\ref{majorizeG}) of $|G|$ remains valid. 
The quantity $1/(|x|-1)$ takes the same values for 
$x=\pm |q_*|^{-\nu -1/2}$, so $|\Theta ^*|>|G|$ 
along $\mathcal{L}_{\nu}(q)$ and $\mathcal{R}_{\nu}(q)$. The double zero $y_j$ 
of $\theta (\bar{q}_j,.)$, $j\in \mathbb{N}$, belongs to the interval 
$\tilde{I}:=[-237,237]$ (with  
$y_1=-2.9\ldots$ and $y_2=2.9\ldots$, see~\cite{KoPRSE2}); 
this follows from $|\bar{q}_1|>1-1/(12\alpha _0)=0.69\ldots$ and 
$\gamma _{13}=89.9\ldots<237$, see Lemma~\ref{lmI} and its proof. 

One has $b_4=-364.1\ldots$, $b_5=-236.7\ldots$, so for $n\geq 5$, $b_n\in \tilde{I}$. For 
$q\in J_n$, there exists a number $|q|^{-\nu -1/2}\in (237,364)$ with 
$\nu \geq \ell _n$ (because $237/364<|\bar{q}_1|$) 
and a line $\mathcal{L}_{\nu}(q)$ (resp. $\mathcal{R}_{\nu}(q)$) 
with no CCPs on and to the left (resp. right) of it. 
By analogy with the case $q\in (0,1)$ one shows that for $q\in (-1,\bar{q}_1]$, 
the moduli of the real parts of the complex zeros of $\theta$ are bounded by 
$|b_4|<364.2$. 
\end{proof}

\end{document}